\documentclass[12pt,draft]{amsart}
\usepackage{amsmath,amssymb,amsbsy,amsfonts,amsthm,latexsym,mathabx,
            amsopn,amstext,amsxtra,euscript,amscd,stmaryrd,mathrsfs,
            cite,array,mathtools,enumerate}
\usepackage[english]{babel}

\newtheorem{theorem}{Theorem}
\newtheorem{lemma}[theorem]{Lemma}

\newtheorem{cor}[theorem]{Corollary}

 \newcommand{\F}{\mathbb{F}}
\newcommand{\K}{\mathbb{K}}

\def\\{\cr}
\def\({\left(}
\def\){\right)}
\def\<{\langle}
\def\>{\rangle}

\def\le{\leqslant}

\newcommand{\Fq}{\mathbb{F}_q}

\newcommand{\Div}[1]{\mathrm{Div}(#1)}
\newcommand{\ord}{\mathrm{ord}}
\renewcommand{\div}[1]{\mathrm{div}(#1)}

\renewcommand{\b}{\mathbf{b}}
\newcommand{\A}{\mathcal{A}}
\newcommand{\B}{\mathcal{B}}
\renewcommand{\P}{\mathbb{P}}
\newcommand{\E}{\mathcal{E}}
\newcommand{\V}{\mathcal{V}}

\newcommand{\x}{\mathbf{x}}
\newcommand{\y}{\mathbf{y}}
\renewcommand{\a}{\pmb{\alpha}}
\renewcommand{\b}{\pmb{\beta}}
\renewcommand{\Im}{\mathrm{Im}}

\newcommand{\cS}{\mathcal{S}}
\newcommand{\cG}{\mathcal{G}}
\newcommand{\cI}{\mathcal{I}}

\begin{document}

\title[Rational function values in small subgroups]
{Values of rational functions in small subgroups of finite fields and the identity testing  problem from powers}

\author{L\'aszl\'o M\'erai} 
\address{ Johann Radon Institute for Computational and Applied Mathematics, Austrian Academy of Sciences,  Altenberger Stra\ss e 69, A-4040 Linz, Austria
}
 \email{laszlo.merai@oeaw.ac.at}

\pagenumbering{arabic}

\subjclass[2000]{11T06, 11Y16, 68Q25}
\keywords{Identity testing from powers, hidden polynomial power, Nullstellensatz, deterministic algorithm}

\begin{abstract}
Motivated by some algorithmic problems, we give low\-er bounds on the size of the multiplicative groups containing rational function images of low-dimensional affine subspaces of a finite field~$\F_{q^n}$ considered as a linear space over a subfield~$\F_q$.
We apply this to the recently introduced algorithmic problem of identity testing of ``hidden'' polynomials $f$ and $g$ over a high degree extension of a  finite field, given oracle access to $f(x)^e$ and $g(x)^e$.

\end{abstract}

\maketitle

\section{Introduction}

Let $\F$ be a finite field. 
For a rational function $r(X)=f(X)/g(X)\in\F(X)$ with two relatively prime polynomials $f(X),g(X)\in \F[X]$ and a set 
$\cS\subset \F$, we use $r(\cS)$ to denote the value set
$$
r(\cS)=\{r(x): \ x\in\cS, g(x)\neq 0 \}\subset \F.
$$

Given $\cS\subset \F$ we consider the smallest possible $e=E_r(\cS)$ such that there is a multiplicative group $\cG$ of $ \F^*$ of size $e=\#\cG$ for which
$$
r(\cS)\subset \cG.
$$

Here, we are mostly interested in the case when $\F=\F_{q^n}$ is a high degree extension of a small field $\F_q$ and $\cS$ is a low-dimensional affine subspace. 

In the extension field $\F_{q^n}$, instead of intervals, we consider the following linear subspace
\begin{equation*}
 \V_m=\{a_0+a_1\alpha+\dots+a_{m-1}\alpha^{m-1}:a_0,a_1,\dots, a_{m-1}\in\Fq \}, \quad 1\leq m\leq n,
\end{equation*}
where $\F_{q^n}=\Fq(\alpha)$, and
we investigate $E_r(\V_m)$, with $r(X)\in\F_{q^n}(X)$.

In the prime field case,  G\'omez-P\'erez and Shparlinski \cite{GP-S15} and Shparlinski \cite{S16} provided lower bounds on $E_f(\cI)$ for intervals $\cI\subset \F_p$ and polynomials $f(X)\in\F_p[X]$.

In this paper we use a quantitative version of an effective Hilbert's Nullstellensatz in function fields given by D'Andrea, Krick and Sombra~\cite{DAKS} to study  this question in a different situation, i.e.\ when the finte field is a high degree extension of a small (fixed)  field and $\cS$ is a low-dimensional affine subspace, see Section~\ref{sec:product_set}. 
We adapt methods of Bourgain, Konyagin, Shparlinski \cite{BKS08}, however new ingredients are needed to handle the function field case.

We apply this result to the \emph{Identity Testing Problem from Powers}. Namely, for a ``hidden'' monic polynomial $f\in\F[X]$,  let $\mathfrak{D}_{e,f}$ be an oracle which on every input $x\in\F$ outputs $\mathfrak{D}_{e,f}(x)=f(x)^e$ for some large positive integer $e\mid \#\F-1$.
Then identity testing problem from powers is:

\begin{quotation}
Given  two oracles $\mathfrak{D}_{e,f}$ and $\mathfrak{D}_{e,g}$ for some unknown monic polynomials $f,g\in\F[X]$, decide whether $f=g$.
\end{quotation}

If $f/g$ is an $(\#\F-1)/e$-th power of a nonconstant rational function, then the oracles $\mathfrak{D}_{e,f}$ and $\mathfrak{D}_{e,g}$ give the same output for each input, therefore it is impossible to distinguish between $f$ and $g$ from the oracles $\mathfrak{D}_{e,f}$ and $\mathfrak{D}_{e,g}$. We write $f\sim_e g$ in this case, and $f \not\sim_e g$ otherwise. 

For polynomials $f\not\sim_e g$ with degree $d$ there is a naive algorithm which calls the oracles on $ed+1$ different inputs and returns $f=g$ if these oracles agree on every input (as for $f\not\sim_e g$, $f(X)^e=g(X)^e$ has at most $ed$ solutions).

For linear polynomials $f(X)=X+s$ and $g(X)=X+t$, this problem is called \emph{hidden shifted power problem}. 
The naive algorithm has been improved by Bourgain, Garaev, Konyagin and Shparlinski~\cite{BGKS12} and Shparlinski~\cite{S14}. 

For prime fields $\F_p$, Ivanyos, Karpinski, Santha, Saxena and Shparlinski \cite{IKSSS} extended the results of  \cite{BGKS12} to arbitrary (non linear) polynomials. For a fixed degree $d$  and $e\rightarrow \infty$ if $e=p^{o(1)}$ one can test whether $f=g$ in time $e^{o(1)}(\log p)^{O(1)}$ in $e^{o(1)}$ oracle calls.

Here we consider the identity testing problem in a different situation, i.e. when the finite field  is an extension of a small (fixed) field with large extension degree. We prove that for a fixed degree $d$ and $e\rightarrow \infty$ if $e=q^{o(n)}$, then one can test whether $f=g$ for polynomials $f(X),g(X)\in\F_{q^n}[X]$ in time $e^{o(1)}(\log q^n)^{O(1)}$ in $e^{o(1)}$ oracle calls, see Section~\ref{sec:indentity}.

\section{Auxiliary results}


Throughout the paper we use the Landau symbol $O$ and the Vinogradov symbol $\ll$. 
Recall that the
assertions $U=O(V)$ and $U \ll V$  are both 
is equivalent to the inequality $|U|\le cV$ with some absolute constant $c>0$.
 To emphasize the dependence of the implied
constant $c$ on some parameter (or a list of parameters) $\rho$, we write $U=O_{\rho}(V)$ or $U \ll_{\rho} V$.

A polynomial $P\in\Fq[T][X_1,\dots, X_n]$ is said to have \emph{height $h$} if $P$ has local degree $h$ in $T$. If $\K$ is a finite extension of $\F_q(T)$, then we define the \emph{height} of $\alpha\in \K$ as the height of its minimal polynomial over $\F_q[T]$.

Clearly, if $\alpha$ and $\beta$ have height at most $h$, then $c \alpha$, $\alpha^{-1}$, $\alpha+\beta$ and $\alpha \cdot \beta$ have height  $O_d(h)$, where $d$  is the degree of the extension $\K/\F_q(T)$.

\subsection{Effective Hilbert's Nullstellensatz}

The next statement is a simplified version of \cite[Theorem~4.22]{DAKS}.

\begin{lemma}\label{lemma:Hilbert}
 Let $P_1,\dots, P_N\in\Fq[T][X_1,\dots, X_n]$ be $N\geq 1$ polynomials of degree  $\deg_{X_1,\dots, X_n}P_i\leq r$ ($i=1,\dots, N$) with height at most $s$. Let $R\in\Fq[T][X_1,\dots, X_n]$ with $t=\max\{1,\deg_{X_1,\dots, X_n}R\}$ and height $u$ such that $R$ vanishes on the variety
 \begin{equation*}
  P_1(X_1,\dots, X_n)=\dots=P_N(X_1,\dots, X_n)=0.
 \end{equation*} 
 
 Then there exist polynomials $Q_1,\dots,Q_N\in\Fq[T][X_1,\dots, X_n]$, $A\in\Fq[T]\setminus\{0\}$ and a positive integer $\mu\in\mathbb{N}$ such that
 \begin{equation*}
  P_1Q_1+\dots+P_NQ_N=A R^{\mu}
 \end{equation*}
and $A$ has height
 \begin{equation*}
\deg_T A\ll  \, r^{\min\{N,n+1 \}}\left( \frac{ u}{t}
+ \sum_{i=1}^{\min\{N,n+1\}}\frac{ s }{\deg_{X_1\dots, X_N}P_i} \right).
 \end{equation*}
 \end{lemma}

 We also need the following result whose proof is identical to the proof of \cite[Lemma~2.14]{Chang}.

\begin{lemma}\label{lemma:Chang}
 Let $P_1,\dots, P_N,R\in\Fq[T][X_1,\dots, X_n]$ be $N+1$ polynomials of degree  $\deg_{X_1,\dots, X_n}P_i\leq r$ ($i=1,\dots, N$), $\deg_{X_1,\dots, X_n}R\leq r$ and height at most $s$.
 If the set
 \begin{equation*}
  P_1(X_1,\dots, X_n)=\dots=P_N(X_1,\dots, X_n)=0 \ \text{and} \ R(X_1,\dots, X_n)\neq 0
 \end{equation*}
is not empty, then it has a point $(\beta_1,\dots, \beta_n)$ in an extension $\K$ of $\Fq(T)$ of degree $[\K:\Fq(T)]\leq \Delta(q,r,N,n)$ such that the height of $\beta_i$ ($i=1,\dots, n$) is at most $C(q,r,n,N)s$ where $\Delta(q,r,n,N)$ and $C(q,r,n,N)$ depend only on $q$, $r$, $n$ and $N$.
\end{lemma}

\subsection{Divisors in function fields}

Let $\K$ be a finite extension of $\F_q(T)$, let $\P_\K$ be the set of places of $\K$ and let $\Div{\K}$ be the set of divisors of $\K$. For a background of divisors of function fields we refer to \cite{Stichtenoth}.

For divisors $D,E\in\Div{\K}$ with
\begin{equation}\label{eq:div}
  D=\sum_{\substack{P\in\P_\K}}n_P P \quad \text{and} \quad  E=\sum_{\substack{P\in\P_\K}}m_PP,
\end{equation}
we say that $D\leq E$ if $n_P\leq m_P$ for all $P$. Specially, $D$ is an \emph{effective} or \emph{positive} divisor, $D\geq 0$, if $n_P\geq 0$ for all $P$.
For a divisor $D\in\Div{\K}$ of form \eqref{eq:div} we define
\begin{equation*}
 D_0=\sum_{\substack{P\in\P_\K\\ n_P>0}}n_P P \quad \text{and} \quad  D_\infty=-\sum_{\substack{P\in\P_\K\\ n_P<0}}n_PP.
\end{equation*}
Clearly, $D=D_0-D_\infty$ with $D_0,D_\infty\geq 0$. Moreover, for divisors $D,E\in\Div{\K}$ of form \eqref{eq:div} let
\begin{equation*}
 \min\{D,E\}=\sum_{P\in\P_\K} \min\{n_P,m_P\} P\in\Div{\K}.
\end{equation*}

The degree of a place $P$ is denoted by $\deg P$ and the degree of a divisor is defined by
\[
 \deg D=\sum_{P\in\P_\K}n_P \cdot \deg P.
\]

For a function $f\in\K$ we define its divisor $(f)$ as
\[
  (f)=\sum_{\substack{P\in\P_\K}}\ord_P(f) P \in\Div{\K},
\]
where $\ord_P(f)$ is the order of $f$ at $P$. 

A divisor $D$ of form \eqref{eq:div} has height $h$ if $h=\max\{|n_P|: \ P\in\P_\K \}$.

\begin{lemma}\label{lemma:height-deg}
Let $\K$ be an extension of $\F_q(T)$ of degree $d=[\K:\Fq(T)]$. If $f\in\K$ has height $h$, then $\deg (f)_\infty$  is at most $2dh$.
\end{lemma}
\begin{proof}
For all places $P\in\P_\K$ we have $\ord_P(T)\leq d$ and $\deg (T)_\infty\leq d$ (see e.g. \cite[I.3.3~Proposition]{Stichtenoth}).
Write
\[
 A_e f^e+\dots+A_1 f +A_0=0 \quad \text{with } A_e,\dots, A_1,A_0\in \F_q[T], \deg A_i \leq h
\]
with $A_e,A_0\neq 0$. 

We have $\deg (A_j)_\infty\leq dh$, thus
\[
 e\cdot \deg (f)_\infty \leq \max_{0\leq j<e}\{j\cdot \deg (f)_\infty + \deg (A_j/A_e)_\infty  \}\leq (e-1)\deg (f)_\infty +2dh,
\]
thus $\deg (f)_\infty\leq 2dh$.
%
\end{proof}

\begin{lemma}\label{lemma:divisor-bound}
Let $\K$ be an extension of $\F_q(T)$ of degree $d=[\K:\Fq(T)]$. The number of effective divisors of degree at most $r$ is at most $d^rq^{2r}$.
\end{lemma}
\begin{proof}
For a place $P\in\P_\K$ let $\pi(P)=P\cap \Fq(T)\in\P_{\Fq(T)}$ be the place of $\Fq(T)$ lies under $P$. By \cite[III.1.12. Corollary]{Stichtenoth},  there are at most $d$ preimages of any place of $\Fq(T)$ under the map $\pi$, moreover  $\deg P\leq d \deg \pi(P)$. We extend the map $\pi:\P_{\K}\rightarrow\P_{\Fq(T)}$ additively to $\pi:\Div{\K}\rightarrow\Div{\Fq(T)}$. The image of every effective $D$ divisor is also effective of degree at most $\deg D$, moreover any divisor of $\Fq(T)$ of degree $r$ has at most $d^r$ preimages under $\pi$.

For any effective divisor $D'$ of $\Fq(T)$, there is a $z\in\Fq(T)$ such that $D=(z)_\infty$

Estimating the number of rational functions whose numerator and denominator have degree at most $r$ by $q^{2r}$ we get that the number of effective divisors of $\K$ of degree at most $r$ is at most $d^r\cdot q^{2r}$.
\end{proof}

For an effective divisor $D \in\Div{\K}$ define $\tau(D)$ as 
\begin{equation*}
 \tau(D)=\#\{E\in\Div{\K} : 0\leq E \leq D\}.
\end{equation*}

We have the following result which is analog to the classical divisor function of integers.

\begin{lemma}\label{lemma:tau}
 Let $\K$ be an extension of $\F_q(T)$ of degree $d=[\K:\Fq(T)]$. For any effective divisor $D\in\Div{\K}$ we have
 \begin{equation*}
  \tau(D)\leq \exp\left(O_{q,d}\left( \, \frac{\deg D}{\log \deg D}\right)  \right).
 \end{equation*}
\end{lemma}

\begin{proof}
We can prove the result in the same line as \cite[Theorem 317]{HardyWright}.

Clearly
\begin{equation*}
 \tau(D)=\prod_{P\in\P_\K}(\ord_P(D)+1).
\end{equation*}

Then for any $\varepsilon>0$, we have
\begin{align*}
 \frac{\tau(D)}{\exp (\varepsilon \, \deg D)}&=\prod_{P\in\P_\K}\frac{\ord_P(D)+1}{\exp(\varepsilon \, \ord_P(D) \deg P)}\\
 &\leq \prod_{P\in\P_\K}\frac{\ord_P(D)+1}{1+\varepsilon \, \ord_P(D) \deg P}.
\end{align*}
For places $P$ with $\deg P>1/\varepsilon$ we have
\begin{equation*}
 \frac{\ord_P(D)+1}{1+\varepsilon \, \ord_P(D) \deg P}< 1,
\end{equation*}
thus we get
\begin{equation*}
 \frac{\tau(D)}{\exp (\varepsilon \, \deg D)}\leq \prod_{\substack{P\in\P_\K\\ \deg P \leq 1/\varepsilon}}\frac{\ord_P(D)+1}{1+\varepsilon \, \ord_P(D) \deg P}.
\end{equation*}

By Lemma \ref{lemma:divisor-bound}, there is a positive constant $c=c(q,d)$ such that there are at most $\exp(c/\varepsilon)$
places of degree at most $1/\varepsilon$. Thus we get
\begin{equation}\label{eq:tau}
 \frac{\tau(D)}{\exp (\varepsilon \, \deg D)}= \left(\frac{1}{\varepsilon}\right)^{\exp(c/\varepsilon)}=\exp (\log (1/\varepsilon) \exp(c/\varepsilon) ).
\end{equation}
Choosing 
\begin{equation*}
 \varepsilon =\frac{2c }{\log \deg D },
\end{equation*}
we get that the logarithm of \eqref{eq:tau} is
$$
\log \frac{\tau(D)}{\exp ( \frac{2c \deg D}{\log \deg D})}\leq \log \log \left(\frac{\deg D}{2c}\right) \cdot (\deg D)^{1/2}\ll\frac{\deg D}{\log \deg D}
$$
if $ \deg D$ is large enough.

\end{proof}


\subsection{Product set in function fields}
We need the following result about the size of product set in function fields. 
For the field of rational numbers see \cite[Lemma~2]{BKS08}, and for number fields see \cite[Lemma~29]{BGKS12}. 

\begin{lemma}
 Let $\K$ be an extension of $\F_q(T)$ of degree $d=[\K:\Fq(T)]$. Let $\A,\B\subset \K$ be finite sets with elements of height at most $h$. Then we have
 \begin{equation*}
  \#(\A\B)>\exp\left(-O_{q,d}\left(\frac{h}{\sqrt{\log h}}\right)\right)\#\A \#\B,
 \end{equation*}
 where the implied constant depends only on $q$ and $d$.
\end{lemma}

\begin{proof}
Put
\begin{equation*}
 \A'=\{\div{a}: \ a\in\A  \} \quad \text{and} \quad \B'=\{\div{b}: \ b\in\B\}.
\end{equation*}
Clearly 
\begin{equation}\label{eq:ini1}
\#\A'\geq \#\A /q, \quad   \#\B'\geq \#\B /q 
\end{equation}
and
\begin{equation}\label{eq:ini2}
 \#(\A\B)\geq \#(\A'+\B').
\end{equation}

In order to prove the result, it is enough to give a lower bound on $\#(\A'+\B')$ in terms of $\#\A'$ and $\#\B'$.

For positive $R$, write 
\[
 \E_R=\{D\in\Div{\K}: \deg D_\infty+ \deg D_0 \leq R \}.
\]
By Lemma \ref{lemma:height-deg} we have $\A', \B'\subset \E_{4dh}$.

Let $\kappa=\kappa(q,d)$ be a positive number, specified later, which may depend only on $q$ and $d$. Denote
\begin{equation*}
 M_1=\frac{ h}{\sqrt{ \log h}} \quad \text{and} \quad  M_2=\exp\left(\kappa\frac{ h}{ \log h}\right).
\end{equation*}

We claim that there is a set $\A_0\subset \E_{4dh}$  of cardinality
\begin{equation}\label{eq:A0-bound}
 \#\A_0> M_2^{-4dh/M_1} \#\A'=\exp\left(-4d \kappa \frac{h}{\sqrt{\log h}}\right)\#\A'
\end{equation}
and $B\in\Div{\K}$ such that $\A_0+B\subset\A'$ and for any effective divisor $E$ of degree $\deg E>M_1$, we have
\begin{equation}\label{eq:A-cond}
 \#\{D\in\A_0: D_0 \geq E\ \text{or} \  D_{\infty}\geq E  \}<\frac{2}{M_2} \#\A_0.
\end{equation}
The construction if straightforward. If $\A_0=\A'$ does not satisfy \eqref{eq:A-cond}, there is $E_1\in\Div{\K}$, $\deg E_1>M_1$ and a subset $\A_1\subset \E_{4dh-\deg E_1} \subset\E_{4dh-M_1}$ of cardinality $\#\A_1\geq M_2^{-1}\#\A'$ and such that 
either $\A_1+E_1\subset \A'$ or $\A_1-E_1\subset \A'$.

Repeat with $\A'$ replaced by $\A_1$ until, after performing $k$ steps, we obtain a subset $\A_k\subset \E_{4dh -kM_1}$ such that $\A_k+B\subset \A'$ for some $B\in\Div{\K}$ and \eqref{eq:A-cond}. Assuming that $\A_k$ is the first set with this property, we derive that
\begin{equation*}
 \#\A_k\geq \frac{1}{M_2}\#\A_{k-1}\geq \dots \geq \frac{1}{M_2^k}\#\A'.
\end{equation*}
Since we obviously have $4dh\geq kM_1$, i.e., $k\leq 4dh/M_1$ which implies~\eqref{eq:A0-bound}.

We now use a similar argument to choose a subset $\B_0\subset\Div{\K}$ of elements of degree at most $h$ of cardinality
\begin{equation*}
  \#\B_0> M_2^{-4dh/M_1} \#\B'=\exp\left(-4d \kappa \frac{h}{\sqrt{\log h}}\right)\#\B'
\end{equation*}
and $B\in\Div{\K}$ such that $\B_0+B\subset\B'$ and for any effective divisor $E$ of degree $E>M_1$, we have
\begin{equation*}
 \#\{D\in\B_0: D_0 \geq E\ \text{or} \  D_{\infty}\geq E  \}<\frac{2}{M_2} \#\B_0.
\end{equation*}
We now establish a lower bound on $\#(\A_0+\B_0)$.

For a given divisor $D\in\E_{4dh}$, denote
\begin{equation*}
 \A_0(D)=\{A\in\A_0: \deg\left(\min\{A_0,D_{\infty}\}\right),  \deg\left(\min\{A_\infty,D_{0}\}\right)\leq M_1\}.
\end{equation*}
Clearly, \eqref{eq:A-cond} implies that, for sufficiently large $h$,
\begin{align*}
 &\#(\A_0\setminus \A_0(D))\\
 &\leq \frac{2}{M_2}\#\A_0\#\{ E\geq 0:  E\leq D_0  \text{ or }  E\leq D_\infty, \ \text{and}\ \deg E >M_1\}\\
 &<\frac{2}{M_2}\#\A_0(\tau(D_0)+\tau(D_\infty))\\
 &<\frac{4}{M_2}\#\A_0\exp\left(O_{q,d}\left( \frac{ h}{\log  h}\right) \right)
\end{align*}
by  Lemma~\ref{lemma:tau}. Thus by an appropriate choice of $\kappa$ we have
\begin{equation*}
 \#(\A_0\setminus \A_0(D))<\frac{1}{2}\#\A_0.
\end{equation*}

Defining $\B_0(D)$ in a similar way, we conclude that
\begin{equation}\label{eq:bound_A0B0}
 \#\A_0(D)>\frac{1}{2}\A_0 \quad \text{and} \quad  \#\B_0(D)>\frac{1}{2}\B_0 
\end{equation}
for every divisor $D\in\E_{4dh}$.

We have
\begin{equation*}
 \#(\A_0+\B_0)\geq \#\left(\bigcup_{A\in\A_0} \{A+B:\ B\in\B_0(A)  \}   \right).
\end{equation*}

Using \eqref{eq:bound_A0B0} we conclude that
\begin{equation}\label{eq:A+B}
 \#(\A_0+\B_0)\geq \frac{1}{2L}\#A_0 \#\B_0,
\end{equation}
where
\begin{equation*}
 L=\max_{D\in\E_{8dh}} \#\{(A,B): A\in\A_0, B\in\B_0(A), A+B=D\}. 
\end{equation*}
It remains to bound $L$. 

Assume that $A+B=D$ for some $A\in \A_0$, $B\in\B_0(A)$ and $D\in\E_{8dh}$ and write
$$
A_0+B_0+D_\infty=A_\infty + B_\infty+D_0.
$$ 
Then we have
\begin{equation}\label{eq:A_0}
A_0\leq B_\infty +D_0.
\end{equation}
Put $J=\min\{A_0,B_\infty\}$ and write $A_0=E+J$. Then $\deg J\leq M_1$ by the definition of $\B_0(A)$ and $E\leq D_0$ by \eqref{eq:A_0}.
Therefore $A_0$ takes at most $\#\{J\in \Div{\K}: \deg J\leq M_1\}\tau(D_0)$ possible values, which is at most $\exp(O_{q,d}( M_1)) \tau(D_0)$ by Lemma~\ref{lemma:divisor-bound}. Similarly, $B_0$ takes at most $\exp(O_{q,d}( M_1)) \tau(D_0)$ possible values while $A_\infty$ and $B_\infty$ take at most $\exp(O_{q,d}( M_1)) \tau(D_\infty)$ possible values.

Therefore by Lemma \ref{lemma:tau} we have
\begin{align}\label{eq:L}
 L&\leq \exp(O_{q,d}( M_1)) \tau(D_0)^2 \tau(D_\infty)^2
 \leq \exp\left(O_{q,d}\left(  \frac{h}{\sqrt{\log h}}\right)  \right),
 \end{align}
provided that $h$ is large enough. Substituting \eqref{eq:L}  in \eqref{eq:A+B} we get the result by \eqref{eq:ini1} and \eqref{eq:ini2}.
\end{proof}

For a set $\A\subset \K$ and positive integer $\nu\in\mathbb{N}$, let $\A^{(\nu)}$ denote the \emph{$\nu$-fold product set}, that is
\begin{equation}\label{eq:nu}
 \A^{(\nu)}=\{a_1\dots a_\nu: \ a_1,\dots, a_\nu\in\A\}.
\end{equation}

\begin{cor}\label{cor:product_set}
Let $\K$ be an extension of $\F_q(T)$ of degree $d=[\K:\Fq(T)]$. Let $\A\subset  \K$ be a finite set of elements of height at most $h$ and let $\nu\in\mathbb{N}$ be a given positive integer.  
Then we have
 \begin{equation*}
  \#\A^{(\nu)}>\exp\left(-O_{q,d,\nu}\left(\frac{h}{\sqrt{\log h}}\right)\right)(\#\A)^{\nu},
 \end{equation*}
 where the implied constant depends only on $q$, $d$ and $\nu$.
\end{cor}

%

\section{Main results}

\subsection{Rational function values in subgroups}\label{sec:product_set}
We note that in order to give a lower bound on $E_r(\V_m)$ it is enough to give a lower bound on the cardinality of $r(\V_m)^{(\nu)}$ for any integer $\nu\geq 1$.

\begin{theorem}\label{lemma:main}
 There is an absolute constant $c>0$ such that if for some fixed integer $\nu\geq 1$, sufficiently large positive integer $m$ and $n$ with
 \begin{equation*}
m\leq \left(\frac{c}{\nu}\right)^{2d+1}n,
 \end{equation*}
then the following holds:
 For any two distinct monic polynomials $f,g\in\Fq[X]$ of degree $d$ the $\nu$-fold product of the set
 \begin{equation*}
  \A=\left\{\frac{f(x)}{g(x)}: \ x\in\V_m  \right\}\subset \F_{q^n}
 \end{equation*}
we have
\begin{equation*}
 \#\A^{(\nu)}>\exp\left(-O_{q,d,\nu}\left(\frac{m}{\sqrt{\log m}}\right)\right) q^{\nu m}.
\end{equation*}
where the implied constant depends only on $q$, $d$ and $\nu$ and $\A^{(\nu)}$ is defined by \eqref{eq:nu}.
\end{theorem}

As $\#\V_m=q^m$, we immediately have the upper bound $\#\A^{(\nu)}\leq q^{\nu m}$. By Theorem~\ref{lemma:main} we also have $\#\A^{(\nu)}\geq q^{vm(1+o(1))}$ as $m\rightarrow\infty$.

\begin{proof} We closely follow the proof of \cite[Lemma~35]{BGKS12}. Let $\kappa$ be a positive absolute constant fixed later and put
\begin{equation*}
 c=\inf_{d\geq 1}\left(\frac{1}{\kappa d^2 2^{2d+2}+1}\right)^{1/(2d+1)}.
\end{equation*}

Let
\begin{equation*}
 f(X)=X^d+\sum_{k=0}^{d-1}a_{d-k}X^k \quad \text{and} \quad g(X)=X^d+\sum_{\ell=0}^{d-1}b_{d-\ell}X^\ell.
\end{equation*}
We move the problem from the finite field  to the function field where we are in the position to apply Lemma \ref{lemma:Hilbert}.
 
Since $\F_{q^n}\cong \Fq[T]/\psi(T)$ for the irreducible polynomial $\psi(T)\in\Fq[T]$ of degree $n$ such that $\psi(\alpha)=0$, we can identify any element $u\in\F_{q^n}$ with the corresponding polynomial $u(T)\in\Fq[T]$ of degree $\deg_T u\leq n-1$. 
 
We consider the collection $\mathcal{P}\subset\Fq[T][\mathbf{U},\mathbf{V}]$, where
\begin{equation*}
 \mathbf{U}=(U_1\dots,U_d) \quad \text{and} \quad  \mathbf{V}=(V_1\dots,V_d)
\end{equation*}
of polynomials
\begin{align*}
 P_{\x,\y}(\mathbf{U},\mathbf{V})=&\prod_{i=1}^{\nu}\left(x_i^d+\sum_{k=0}^{d-1}U_{d-k}x_i^k  \right)\left(y_i^d+\sum_{\ell=0}^{d-1}V_{d-\ell}y_i^\ell  \right)\\
 &\quad - \prod_{i=1}^{\nu}\left(x_i^d+\sum_{k=0}^{d-1}V_{d-k}x_i^k  \right)\left(y_i^d+\sum_{\ell=0}^{d-1}U_{d-\ell}y_i^\ell  \right),
\end{align*}
where $\x=(x_1(T),\dots, x_\nu(T))$ and $\y=(y_1(T),\dots, y_\nu(T))$ with polynomial entries  $x_i(T),y_i(T)\in\Fq[T]$ ($1\leq i\leq\nu$) of degree at most $m-1$ such that
\begin{equation*}
  P_{\x,\y}(a_1(T),\dots, a_d(T),b_1(T),\dots, b_d(T))\equiv 0 \mod \psi(T).
\end{equation*}
This is equivalent to
\begin{multline*}
   P_{\x,\y}(a_1(\alpha),\dots, a_d(\alpha),b_1(\alpha),\dots, b_d(\alpha))\\
   = \prod_{i=1}^\nu f(x_i(\alpha))g(y_i(\alpha))-  \prod_{i=1}^\nu f(y_i(\alpha))g(x_i(\alpha)).
\end{multline*}

Clearly, if $P_{\x,\y}(\mathbf{U},\mathbf{V})$ is identical to zero modulo $\psi(T)$, then, by the uniqueness of polynomial factorization in the ring $\F_{q^n}[\mathbf{U},\mathbf{V}]$, we see that for every $i=1,\dots, \nu$, for every linear form
\begin{equation*}
 L_{x_i}(\mathbf{U})=x_i(\alpha)^d+U_{d-1}x_i(\alpha)^{d-1}+\dots+U_1x_i(\alpha)+U_0
\end{equation*}
there should be an equal (over $\F_{q^n}$) form
\begin{equation*}
  L_{y_j}(\mathbf{U})=y_j(\alpha)^d+U_{d-1}y_j(\alpha)^{d-1}+\dots+U_1y_j(\alpha)+U_0
\end{equation*}
with some $j=1,\dots, \nu$. Hence, if $P_{\x,\y}(\mathbf{U},\mathbf{V})$ vanishes, then $\x$ and $\y$ can be obtained from each other by a permutation of their components. 

Therefore, if $\mathcal{P}$ contains only the zero polynomial, then each $\lambda\in\F_{q^n}$, given by the product
\begin{equation*}
 \lambda= \prod_{i=1}^\nu \frac{\displaystyle x_i(\alpha)^d+\sum_{k=0}^{d-1}a_{d-k}x_i(\alpha)^k}{\displaystyle x_i(\alpha)^d+\sum_{k=0}^{d-1}b_{d-k}x_i(\alpha)^k}  
\end{equation*}
appears no more that $\nu!$ times. In turn this implies that
\begin{equation*}
 \#\A^{(\nu)}\geq \frac{1}{\nu!}(\#\A)^{\nu}\gg q^{\nu m}.
\end{equation*}

Thus we can assume that $\mathcal{P}$ contains non-zero polynomials.
 
Clearly, each polynomial $P(\mathbf{U},\mathbf{V})\in\mathcal{P}$ is of total degree $\nu$ in $\mathbf{U}$ and $\mathbf{V}$ and it is of degree at most $2d\nu m$ in $T$. 

We take a family $\mathcal{P}_0$ containing the largest possible number
\begin{equation*}
 N\leq q^{2\nu m}-1
\end{equation*}
of linearly independent polynomials $P_1,\dots,P_N\in\mathcal{P}$ over $\F_q(T)$ and consider the variety
\begin{equation*}
 \V: \ \{(\mathbf{U},\mathbf{V})\in\overline{\Fq(T)}^{2d}: \ P_1(\mathbf{U},\mathbf{V})=\dots=P_N(\mathbf{U},\mathbf{V})=0 \}.
\end{equation*}
Clearly, $\V\neq\emptyset$ as it contains the diagonal $\mathbf{U}=\mathbf{V}$.

We claim that $\V$ contains a point outside the diagonal, that is, there is a point $(\a,\b)$ with $\a,\b\in\overline{\Fq(T)}^{d}$ and $\a\neq\b$.

Assume that $\V$ does not contain a point outside of the diagonal. Then for every $k=1,\dots, d$, the polynomial
\begin{equation*}
 R_k(U_1,\dots, U_d,V_1,\dots, V_d)=U_k-V_k
\end{equation*}
vanishes on $\V$.

Then by Lemma~\ref{lemma:Hilbert}, there are polynomials $Q_{k,1},\dots, Q_{k,N}\in\Fq[T][\mathbf{U},\mathbf{V}]$, $A_k\in\Fq[T]$ and positive integer $\mu_k$ with
\begin{equation}\label{eq:A-bound}
 \deg_T A_k\leq c_0 \nu d^2 (2\nu)^{2d} m
\end{equation}
for some absolute constant $c_0$ and such that
\begin{equation}\label{eq:hilbert-eq}
 P_1Q_{k,1}+\dots+P_NQ_{k,N}=A_k(U_k-V_k)^{\mu_k}.
\end{equation}
Since $f\neq g$, there is a $k\in\{1,\dots, d\}$ for which $a_k(T)\not\equiv b_k(T)\mod \psi(T)$. For this $k$ we substitute
\begin{equation*}
 (\mathbf{U},\mathbf{V})=(a_1(T),\dots, a_d(T),b_1(T),\dots, b_d(T))
\end{equation*}
in \eqref{eq:hilbert-eq}. Recalling the definition of the set $\mathcal{P}$ we now derive that $\psi(T)\mid A_k(T)$ and thus $\deg_T A_k\geq n$. Then
choosing $\kappa=c_0$, \eqref{eq:A-bound} violates the first condition of the theorem.
Hence the set
\begin{equation*}
 \mathcal{U}=\V\cap [\mathbf{U}-\mathbf{V}\neq 0]
\end{equation*}
is nonempty. Applying Lemma~\ref{lemma:Chang} we see that it has a point $(\a,\b)$ with components of height at most $C(q,d,\nu) m$  in an extension $\K$ of $\Fq(T)$ of degree $[\K:\Fq(T)]\leq \Delta(q,d,\nu)$, where $C(q,d,\nu)$ and $\Delta(q,d,\nu)$ depend only on $q$, $d$ and $\nu$.

Consider the maps $\Phi:\V_m^\nu\rightarrow \F_{q^n}$ given by
\begin{equation*}
 \Phi: \x=(x_1,\dots,x_\nu)\mapsto \prod_{j=1}^{\nu}\frac{f(x_j)}{g(x_j)}
\end{equation*}
and $\Psi:\V_m^\nu\rightarrow \K$ given by
\begin{equation*}
 \Psi: \x=(x_1,\dots,x_\nu)\mapsto \prod_{j=1}^{\nu}\frac{F_{\a}(x_j(T))}{G_{\b}(x_j(T))},
\end{equation*}
where
\begin{equation*}
 F_{\a}(X)=X^d+\sum_{k=0}^{d-1}\alpha_{d-k}(T)X^k \quad \text{and} \quad  G_{\b}(X)=X^d+\sum_{\ell=0}^{d-1}\beta_{d-\ell}(T)X^\ell.
\end{equation*}
By construction of $(\a,\b)$ we have that $\Psi(\x)=\Psi(\y)$ if $\Phi(\x)=\Phi(\y)$. Hence
\begin{equation*}
 \#\A^{(\nu)}\geq \Im \Psi=\#\mathcal{C}^{(\nu)},
\end{equation*}
where $\Im \Psi$ is the image set of the map $\Psi$ and
\begin{equation*}
 \mathcal{C}=\left\{ \frac{F_{\a}(x)}{G_{\b}(x)}: \ x\in\Fq[T], \deg x< m \right\}\subset \K.
\end{equation*}
Using Corollary~\ref{cor:product_set}, we derive the result.
\end{proof}

\subsection{The identity testing problem}\label{sec:indentity}
Here we give an application of Theorem~\ref{lemma:main} to the identity testing problem.

\begin{theorem}\label{thm:1}
Let $q$ be a fixed prime power and let $e$ and $n$ be positive integers with $e\mid q^n-1$ and $e\leq q^{\delta n}$ for some fixed $\delta$. Given two oracles $\mathfrak{D}_{e,f}$ and $\mathfrak{D}_{e,g}$ for some unknown monic polynomials $f,g\in\F_{q^n}[X]$ of degree $d$ with $f\not\sim_e g$, there is a deterministic algorithm to decide whether $f=g$ in at most $e^{O_d\left(\delta^{1/(2d)}\right)}$ queries to the oracles   $\mathfrak{D}_{e,f}$ and $\mathfrak{D}_{e,g}$. Here the implied constant might depend on $d$.
\end{theorem}
%

\begin{proof}
We set 
\begin{equation*}
 \nu=\left \lfloor \frac{c^{1+1/(2d)}}{(2\delta)^{1/(2d)}} \right \rfloor \quad \text{and} \quad m= \left \lfloor \frac{2 \log e}{\nu \log q} \right \rfloor 
\end{equation*}
where $c$ is the constant of Theorem~\ref{lemma:main}. We note that
\begin{equation*}
 \frac{2\delta}{\nu}\leq \left(\frac{c}{\nu}\right)^{2d+1}
\end{equation*}
so we have
\begin{equation}\label{eq:m}
 m\leq  \frac{2 \log e}{\nu \log q} \leq \frac{2 \delta n}{\nu }\leq \left(\frac{c}{\nu}\right)^{2d+1}n. 
\end{equation}

We now query the oracles $\mathfrak{D}_{e,f}$ and $\mathfrak{D}_{e,g}$ for all $x\in\V_m$.

If the oracles return two distinct values, then clearly $f\neq g$. Now assume
\begin{equation*}
 f(x)^e=g(x)^e,\quad x\in\V_m.
\end{equation*}
Therefore, the values $f(x)/g(x)$, $x\in\V_m$ belong to the subgroup $\mathcal{G}_e$ of $\F_{q^n}^*$ of order $e$. Hence for the set
\begin{equation*}
 \A=\left\{\frac{f(x)}{g(x)}:\ x\in\V_m \right\}\subset \F_{q^n}
\end{equation*}
for any integer $\nu\geq 1$ we have
\begin{equation}\label{eq:subgroup}
 \A^{(\nu)}=\{a_1\dots a_\nu: a_1,\dots, a_\nu\in\A\}\subset \mathcal{G}_e
\end{equation}
thus $\A^{(\nu)}\leq e$. By \eqref{eq:m}, Theorem~\ref{lemma:main} yields $e\geq q^{ m(\nu +o(1))}$ which contradicts \eqref{eq:subgroup} since we have $\nu m > (2+o(1))\log e/\log q$ as $e\rightarrow \infty$ for the above choice of parameters.   We also note that with the above choice of $\nu$ we have $m\leq c_0 \frac{\delta^{1/(2d)}\log e}{\log q}$ for an absolute constant $c_0$, thus
\begin{equation*}
 \#\V_m=q^m\leq e^{c_0 \delta^{1/(2d)}}
\end{equation*}
which proves the result.
\end{proof}
\section*{Acknowledgment}

The author would like to thank Igor Shparlinski for suggesting the problem and
for his helpful comments. The author also thanks Arne Winterhof for the fruitful discussions.

The author is partially supported by the Austrian Science Fund (FWF): Project P31762.

\end{document}